\documentclass[11pt,a4paper]{article}
\usepackage[margin=2.6cm]{geometry}
\usepackage{amsmath,amssymb,amsthm}
\usepackage{mathtools}
\usepackage{enumitem}
\usepackage{hyperref}
\hypersetup{colorlinks=true,linkcolor=blue,citecolor=blue,urlcolor=blue}
\usepackage[utf8]{inputenc}
\newtheorem{theorem}{Theorem}
\newtheorem{lemma}{Lemma}
\newtheorem{corollary}{Corollary}

\newtheorem{remark}{Remark}

\newcommand{\HH}{\mathbb{H}}
\newcommand{\RR}{\mathbb{R}}
\newcommand{\Ric}{\mathrm{Ric}}
\newcommand{\Lie}{\mathcal{L}}
\newcommand{\slR}{\mathfrak{sl}(2,\RR)}

\title{Ricci Soliton Classification on $\HH^2\times\RR$}
\author{Anton Khaliapin}
\date{August, 2026 }

\begin{document}
\maketitle

\begin{abstract}
We study Ricci solitons on the Riemannian manifold
$\mathbb H^2\times\mathbb R$ equipped with the standard metric.
A complete classification of soliton vector fields is obtained: they form
a four-dimensional affine space, namely a translate of
the Killing algebra $\mathfrak{isom}(\mathbb H^2\times\mathbb R)$. All corresponding solitons are expanding. In addition, gradient solitons are fully
characterized and shown to form a one-parameter subfamily of the complete family of soliton fields.
As a byproduct, every soliton vector field turns out to be affine, preserving the
Levi-Civita connection, the curvature tensor, and the Ricci tensor.
\end{abstract}

\section{Introduction}

Homogeneous three-dimensional Riemannian manifolds occupy a central place in
the geometry of low-dimensional manifolds, primarily due to Thurston's geometrization picture, which singles out eight model geometries
$E^3,\ H^3,\ S^3,\ S^2\times\RR,\ \mathrm{Nil},\ \widetilde{SL_2\RR},
\ \mathrm{Sol}$, and $\HH^2\times\RR$ \cite{thurston1997}. In this context, a model geometry is understood as a smooth simply connected manifold on which some Lie group acts transitively and the stabilizer is compact.

For the manifold $\HH^2 \times \RR$, the connected component of the isometry group is $\operatorname{PSL}(2,\RR)\times \RR$ with stabilizer $\operatorname{SO}(2)$.
The Bianchi type~III group is a three-dimensional subgroup of \(\operatorname{PSL}(2,\mathbb R)\times\mathbb R\) that acts transitively on the manifold, and its stabilizer is trivial. Thus, $\HH^2 \times \RR$ can be realized as a group manifold with a left-invariant metric.

A Ricci soliton is a Riemannian manifold with metric $g$ such that there exists a vector field $X$ and a constant $\lambda\in\RR$ satisfying the equation
\begin{equation}
\Lie_Xg+\Ric=\lambda g. \label{eq:soliton-def}
\end{equation}
They are of interest since they are self-similar solutions of the Ricci flow and serve as local models for singularity formation \cite{hamilton1988,hamilton1982}.
A soliton is called \emph{shrinking}, \emph{steady}, or \emph{expanding}
depending on the sign of $\lambda$: $\lambda>0$, $\lambda=0$, or $\lambda<0$; it is a
\emph{gradient} Ricci soliton if $X=\operatorname{grad}h$ for some
smooth function $h$ \cite{cao2009}.

The problem of left-invariant Ricci solitons on $\HH^2 \times \RR$ was recently investigated in \cite{belarbi2020}, where the system of partial differential equations for the components of $X$ was solved and the non-existence of gradient solitons was claimed. We revisit this computation and find that it misses a family of solutions; in particular, a one-parameter gradient subfamily exists.
We also note that $\HH^2\times\RR$ has been studied within the more restrictive notion of an algebraic Ricci soliton, for which the Ricci operator decomposes as $c\,\mathrm{Id}+D$, where $D$ is a derivation of the Lie algebra \cite{atashpeykar2020}; this condition is an algebraic constraint on the operator and does not address the question of whether the soliton vector field is gradient, which is the subject of the present article.

In this article, we show that
the classification on $\HH^2\times\RR$ admits a short, geometrically
transparent derivation: since the $\HH^2$-factor has constant curvature
$-1$, the soliton equation forces the $\HH^2$-part of $X$ to be
a \emph{Killing field} on $\HH^2$ (Lemma~\ref{lem:rigidity} below);
the algebra of such fields is classical and is isomorphic to $\slR$. Together with
the component along $\RR$, this yields the complete four-parameter
classification (Theorem~\ref{thm:main}): the set of soliton fields
forms a four-dimensional affine space that is an affine
translate of the Killing algebra $\mathfrak{isom}(\HH^2\times\RR)\simeq\slR\oplus\RR$,
together with an explicit gradient subfamily (Theorem~\ref{thm:gradient}).

\section{Geometry of \texorpdfstring{$\HH^2\times\RR$}{H2 x R}}

Let $\HH^2=\{(x,y)\in\RR^2:y>0\}$ carry the hyperbolic metric
$g_{\HH^2}=\frac1{y^2}(dx^2+dy^2)$. With the group law obtained from
the composition of proper affine maps of the line, $\HH^2$ is a Lie group,
and $g_{\HH^2}$ is left-invariant; consequently, so is the product metric
on $\HH^2\times\RR$,
\begin{equation}
g=\frac1{y^2}\bigl(dx^2+dy^2\bigr)+dz^2. \label{eq:metric}
\end{equation}
An orthonormal frame of left-invariant vector fields is given by
\begin{equation}
E_1=y\,\partial_x,\qquad E_2=y\,\partial_y,\qquad E_3=\partial_z,
\label{eq:frame}
\end{equation}
with Lie brackets
\begin{equation}
[E_1,E_2]=-E_1,\qquad [E_2,E_3]=0,\qquad [E_3,E_1]=0. \label{eq:brackets}
\end{equation}
A direct computation from \eqref{eq:brackets} using the Koszul formula \cite{docarmo} gives
the Levi-Civita connection
\begin{equation}\label{eq:conn}
\nabla_{E_1}E_1=E_2,\ \ \nabla_{E_1}E_2=-E_1,\ \ \nabla_{E_1}E_3=0,\qquad
\nabla_{E_2}E_i=0,\qquad \nabla_{E_3}E_i=0\ (i=1,2,3),
\end{equation}
the curvature $R(E_1,E_2)E_1=E_2,\ R(E_1,E_2)E_2=-E_1$, and the Ricci tensor
\begin{equation}
\Ric_{11}=\Ric_{22}=-1,\qquad \Ric_{33}=0,\qquad \Ric_{ij}=0\ (i\neq j).
\label{eq:ricci}
\end{equation}
Thus, in the tangent directions $E_1,E_2$,
the Ricci tensor is exactly equal to $-g_{\HH^2}$, reflecting the fact
that $(\HH^2,g_{\HH^2})$ is an Einstein manifold of constant curvature $-1$;
on the $\RR$-factor it vanishes.

\section{The soliton equation}

For $X=f_1E_1+f_2E_2+f_3E_3$ with $f_1,f_2,f_3\in C^\infty(\HH^2\times\RR)$,
the orthonormality of $\{E_i\}$ reduces the Lie derivative
$(\Lie_Xg)(Y,Z)=X\bigl(g(Y,Z)\bigr)-g([X,Y],Z)-g(Y,[X,Z])$ to
\begin{equation}
\begin{cases}
(\Lie_Xg)(E_1,E_1)=2\bigl(E_1f_1-f_2\bigr), &
(\Lie_Xg)(E_1,E_2)=f_1+E_1f_2+E_2f_1,\\
(\Lie_Xg)(E_1,E_3)=E_1f_3+E_3f_1, &
(\Lie_Xg)(E_2,E_2)=2\,E_2f_2,\\
(\Lie_Xg)(E_2,E_3)=E_2f_3+E_3f_2, &
(\Lie_Xg)(E_3,E_3)=2\,E_3f_3.
\end{cases}
\end{equation}
Since
$E_1f=y\,\partial_xf$, $E_2f=y\,\partial_yf$, and $E_3f=\partial_zf$ from
\eqref{eq:frame}, the soliton equation \eqref{eq:soliton-def} together with
\eqref{eq:ricci} is equivalent to
\begin{equation}
\begin{cases}
2y\,\partial_xf_1-2f_2-1=\lambda, & (a)\\
f_1+y\,\partial_yf_1+y\,\partial_xf_2=0, & (b)\\
2y\,\partial_yf_2-1=\lambda, & (c)\\
\partial_zf_1+y\,\partial_xf_3=0, & (d)\\
\partial_zf_2+y\,\partial_yf_3=0, & (e)\\
2\,\partial_zf_3=\lambda. & (f)
\end{cases}\label{eq:system}
\end{equation}

\section{Classification via Killing fields}

The observation is that the subsystem formed by
equations (a),(b),(c) of \eqref{eq:system} involves only $f_1,f_2$,
that is, only the $\HH$-component $Y:=f_1E_1+f_2E_2$ of the field $X$, and is precisely
the Ricci soliton equation on $(\HH^2,g_{\HH^2})$:
\begin{equation}
\Lie_Yg_{\HH^2}+\Ric_{\HH^2}=\lambda\,g_{\HH^2},
\label{eq:H2-soliton}
\end{equation}
provided that $f_1,f_2$ are independent of $z$. We first verify the latter
point and then classify the solutions of \eqref{eq:H2-soliton} using the constancy of the curvature.

\begin{lemma}\label{lem:zindep}
Every solution $(f_1,f_2,f_3,\lambda)$ of the system \eqref{eq:system} has
$f_1,f_2$ independent of $z$.
\end{lemma}

\begin{proof}
From (f) we have
\[
f_3=\frac{\lambda}{2}z+\xi(x,y).
\]
Equations (d) and (e) give
\[
\partial_zf_1=-y\xi_x,\qquad
\partial_zf_2=-y\xi_y.
\]
Therefore,
\[
f_1=\varphi z+\psi,\qquad
f_2=\rho z+\chi,
\]
where
$$
\varphi=-y\xi_x,\qquad
\rho=-y\xi_y.
$$
Substituting into (c) and comparing the coefficients
of $z$ gives
$\rho_y=0.$
Hence $\rho=\rho(x)$ and $\xi=-\rho(x)\ln y+d(x),$
whence $\varphi=y\rho'(x)\ln y-yd'(x).$
The coefficient of $z$ in (b) equals
$$
\rho'(x)\ln y+\rho'(x)-d'(x)=0.
$$
Since this identity holds for all $y>0$ and the
functions $\ln y$ and $1$ are linearly independent,
we obtain $\rho'(x)=0$ and $d'(x)=0$.
Thus, $\rho=c_0$ and $\varphi=0$. Finally, the
coefficient of $z$ in (a) equals $-2c_0$, whence
$c_0=0$. Therefore,
$$
\partial_zf_1=\partial_zf_2=0.
$$
\end{proof}

\begin{lemma}\label{lem:rigidity}
Let $Y$ be a smooth vector field on $\HH^2$ such that
$\Lie_Yg_{\HH^2}=cg_{\HH^2}$ for a constant $c\in\RR$. Then $c=0$, that is,
$Y$ is a Killing field.
\end{lemma}

\begin{proof}
Let $\{\phi_t\}$ be the local flow generated by $Y$. The condition
$\Lie_Y g_{\HH^2}=c\,g_{\HH^2}$
implies $\phi_t^*g_{\HH^2}=e^{ct}g_{\HH^2}.$
Since $\phi_t$ is a diffeomorphism and $g_{\HH^2}$ has constant Gaussian
curvature $\mathcal K=-1$, the pullback metric $\phi_t^*g_{\HH^2}$ also has
constant Gaussian curvature equal to $\mathcal K=-1$. On the other hand, a constant
rescaling of a Riemannian metric by a factor $\alpha>0$ changes its
Gaussian curvature according to
$\mathcal K_{\alpha g}=\alpha^{-1}\mathcal K_g.$
Hence $\mathcal K_{e^{ct}g_{\HH^2}}=-e^{-ct}.$
Comparing the two expressions gives $-e^{-ct}=-1$,
so $c=0$.
\end{proof}

Combining Lemmas~\ref{lem:zindep} and \ref{lem:rigidity} with
\eqref{eq:H2-soliton}: writing $\Ric_{\HH^2}=-g_{\HH^2}$, equation
\eqref{eq:H2-soliton} takes the form $\Lie_Yg_{\HH^2}=(\lambda+1)g_{\HH^2}$,
so Lemma~\ref{lem:rigidity} gives
\begin{equation}
\lambda=-1 \qquad\text{and}\qquad \Lie_Yg_{\HH^2}=0,
\end{equation}
that is, $Y$ is a Killing field on $\HH^2$. The Killing algebra on $\HH^2$
is classical: it is isomorphic to $\slR$ and is realized by the three vector fields
\begin{equation}
V_1 = \partial_x,\qquad V_2 =  x\partial_x+y\partial_y,\qquad
V_3 =(x^2-y^2)\partial_x+2xy\partial_y
\end{equation}
generating, respectively, translations,
dilations, and the one-parameter group of ``special conformal''
maps $z\mapsto z/(cz+1)$ of the boundary circle (here $z=x+iy$). Consequently, in the coordinate basis
\begin{equation}
Y=\bigl[A(x^2-y^2)+Bx+C\bigr]\partial_x+\bigl[2Axy+By\bigr]\partial_y,
\qquad A,B,C\in\RR, \label{eq:Y-killing}
\end{equation}
and correspondingly $f_1=Y^x/y,\ f_2=Y^y/y$, that is,
\begin{equation}
f_1=\frac{1}{y} \left( A(x^2-y^2)+Bx+C \right),\qquad f_2=2Ax+B. \label{eq:f12}
\end{equation}
It remains to determine $f_3$. For $\lambda=-1$, equation (f) gives
$f_3=E_0-\tfrac z2+\xi(x,y)$ for a constant $E_0$ and a function $\xi$
independent of $z$; substituting \eqref{eq:f12} (now independent of $z$)
into (d) and (e) gives $\partial_x\xi=\partial_y\xi=0$, so $\xi$ is itself
a constant, which is absorbed into $E_0$. We state the resulting
classification.

\begin{theorem}[Classification]\label{thm:main}
The manifold $(\HH^2\times\RR,\, g)$ with the metric \eqref{eq:metric} is an expanding Ricci soliton with $\lambda=-1$. The set $\mathcal S$ of all vector fields $X$ satisfying
the Ricci soliton equation \eqref{eq:soliton-def} is a
four-dimensional affine space and has the form
\begin{equation}
    \mathcal{S} = -\frac{1}{2}z\partial_z + \mathfrak{isom}(\HH^2 \times \RR).
    \label{eq:4_param_set}
\end{equation}
\end{theorem}

\begin{remark}
In \eqref{eq:4_param_set} the sum is understood as an affine translate of the vector space of Killing fields $\mathfrak{isom}(\HH^2 \times \RR) \simeq \slR \oplus \RR.$ The field $X$ from \eqref{eq:4_param_set} has the form
\begin{equation}
    X = A V_3\, + \, BV_2\, + \, CV_1\,+\,E_0 \partial_z\, -\, \frac{1}{2}z \partial_z  ,\qquad A,B,C,E_0\in\RR.
    \label{eq:X-final}
\end{equation} The three parameters $A,B,C$ correspond to an arbitrary Killing field on $\HH^2$, while the parameter $E_0$ corresponds to an arbitrary Killing component along $\mathbb R$. Thus, the soliton field $X$ differs from a Killing field on $\HH^2 \times \RR$ by the fixed component $-\frac 12 z\partial_z$.
Equivalently, $\mathcal S$ is an affine hyperplane
in the five-dimensional algebra $\slR\oplus\mathfrak{aff}(\RR)$.
\end{remark}

\begin{corollary}
\label{cor:Lie}
For any soliton vector field of the form
\eqref{eq:X-final}, the Lie derivative of the metric has the form
\begin{equation*}
    \mathcal L_Xg=-dz^2.
\end{equation*}
In particular, $\mathcal L_Xg$ is independent of the choice of soliton
vector field within the affine family $\mathcal S$.
\end{corollary}

\begin{remark}
For any soliton vector field from the family
\eqref{eq:X-final}, one has
$\mathcal{L}_X\nabla = 0$. Indeed, for the Levi-Civita connection,
the Lie derivative of the connection is given by the standard formula
$$
g\bigl((\mathcal L_X \nabla)(Y,Z),\, W\bigr)
= \frac{1}{2}\, \Bigl(
  \nabla_Y(\mathcal L_X g)(Z,W)
  + \nabla_Z(\mathcal L_X g)(Y,W)
  - \nabla_{W}(\mathcal L_X g)(Y,Z)
\Bigr).
$$
Since $\mathcal L_X g = -dz^2$ and $dz$ is parallel ($\nabla dz=0$),
we have $\nabla(\mathcal L_X g)=0$; hence $\mathcal L_X\nabla=0$.
Consequently, the curvature tensor and the Ricci tensor are also
invariant under $X$: $\mathcal L_X R=0$ and $\mathcal L_X \operatorname{Ric}=0$.
\end{remark}

\section{Gradient Ricci solitons}

\begin{theorem}\label{thm:gradient}
Among the vector fields \eqref{eq:X-final}, the one-parameter
subfamily of the form
\begin{equation}
X=\Bigl(E_0-\frac z2\Bigr)\partial_z=\operatorname{grad}h
\end{equation}
is a gradient Ricci soliton.
The corresponding potential has the form
\begin{equation*}
    h(z)=E_0z-\frac{z^2}{4}+\mathrm{const}.
\end{equation*}
In other words, the soliton field \eqref{eq:X-final} is gradient
if and only if its $\HH^2$-component vanishes.
\end{theorem}
\begin{proof}
By Theorem~\ref{thm:main}, any soliton field has the form
$$
X=K+\left(E_0-\frac z2\right)\partial_z,
\qquad
K\in\mathfrak{sl}(2,\RR).
$$
Suppose that $X=\operatorname{grad}h$. Then
$$
h=h_{\HH}+E_0z-\frac{z^2}{4}
$$
for some function $h_{\HH}$ on $\HH^2$, and therefore
$$
K=\operatorname{grad}_{\HH^2}h_{\HH}.
$$
Since $K$ is a Killing field,
$$
\mathcal L_Kg_{\HH^2}
=2\operatorname{Hess}_{\HH^2}h_{\HH}=0,
$$
that is, $\nabla\,dh_{\HH}=0$. Since the metric is
parallel, lowering the index commutes with $\nabla$,
so this is equivalent to $\nabla K=0$.

In particular, $R(U,V)K=0$
for all $U,V\in T\HH^2$. But $\HH^2$ has constant sectional
curvature $-1$, so
$$
R(U,V)K =-\bigl(g_{\HH^2}(V,K)U-g_{\HH^2}(U,K)V \bigr).
$$
Choosing $U=K$ and $V\perp K$ at a point where $K\neq0$, we obtain
$$R(K,V)K=|K|^2V=0,$$
a contradiction. Therefore, $K\equiv0$.
Thus, $A=B=C=0$. The converse is obvious.
\end{proof}

\begin{remark}
The gradient subfamily corresponds to the trivial component of $X$:
this is the Killing field $Y\equiv0$ on the $\HH^2$-factor together with
the one-dimensional gradient soliton potential on the flat $\RR$-factor.
The entire $\slR$-family
of Killing components is --- by Lemma~\ref{lem:rigidity} --- necessarily
non-gradient, since a nonzero Killing field
is not a gradient field.
\end{remark}

\section{Discussion and conclusion}

The classification obtained shows that all Ricci solitons on $\HH^2\times\mathbb R$ are expanding with $\lambda=-1$, and their vector fields form a four-parameter affine family determined by the Killing fields of the hyperbolic factor and a component along $\RR$. Gradient solitons constitute only a one-parameter subfamily and arise exclusively when the $\HH^2$-component vanishes.
Thus, the geometry of constant negative curvature completely fixes the hyperbolic part of the soliton and reduces the classification to the Killing algebra $\mathfrak{isom}(\HH^2\times\RR)\simeq\slR\oplus\RR$, translated by the fixed direction $-\tfrac12z\partial_z$.

\end{document}